\documentclass{amsart}

\usepackage[outer=1.25in, inner=1.25in, top=1.25in, bottom=1.25in]{geometry}
\usepackage{amsmath}
\newtheorem{theorem}{Theorem}[section]

\newtheorem{lemma}[theorem]{Lemma}

\usepackage{hyperref,amssymb,xcolor}
\begin{document}
	
	\title[Base size and minimal degree]{On the base size and minimal degree of transitive groups}
	
	\author[L. Guerra]{Lorenzo Guerra}
	\address{Lorenzo Guerra, Dipartimento di Matematica Pura e Applicata,\newline
		University of Milano-Bicocca, Via Cozzi 55, 20126 Milano, Italy} 
	\email{l.guerra@unimib.it}

	\author[A. Mar\'oti]{Attila Mar\'oti}
	\address{Attila Mar\'oti, Hun-Ren Alfr\'ed R\'enyi Institute of Mathematics, Re\'altanoda Utca 13-15, H-1053, Budapest, Hungary}
	\email{maroti@renyi.hu}

	\author[F. Mastrogiacomo]{Fabio Mastrogiacomo}
	\address{Fabio Mastrogiacomo, Dipartimento di Matematica ``Felice Casorati", University of Pavia, Via Ferrata 5, 27100 Pavia, Italy} 
	\email{fabio.mastrogiacomo01@universitadipavia.it}

	\author[P. Spiga]{Pablo Spiga}
	\address{Pablo Spiga, Dipartimento di Matematica Pura e Applicata,\newline
		University of Milano-Bicocca, Via Cozzi 55, 20126 Milano, Italy} 
	\email{pablo.spiga@unimib.it}
	\keywords{base size, minimal degree, transitive, $p$-group, multinomial coefficient}        
	\thanks{The first three authors are members of the GNSAGA INdAM research group and kindly acknowledge their support.\\ The first and fourth authors are funded by the European Union via the Next
		Generation EU (Mission 4 Component 1 CUP B53D23009410006, PRIN 2022, 2022PSTWLB, Group
		Theory and Applications). The second author was supported by the National Research, Development and Innovation Office (NKFIH) Grant No.~K138596, No.~K132951 and Grant No.~K138828.}
	\maketitle	
	\begin{abstract}  
		Let $G$ be a permutation group, and denote with $\mu(G)$ and $b(G)$ its minimal degree and base size respectively. We show that for every $\varepsilon>0$, there exists a transitive permutation group $G$ of degree $n$ with
		\[
			\mu(G)b(G) \geq n^{2-\varepsilon}.
		\]
		We also identify some classes of transitive and intransitive groups whose base size and minimal degree have a smaller upper bound, shared with primitive groups.
	\end{abstract}
\section{Introduction}\label{sec:intro}
	Let $G$ be a finite permutation group acting on a set $\Omega$. The \textit{\textbf{minimal degree}} of $G$, denoted by $\mu(G)$, is the smallest number of elements of $\Omega$ that are moved by any non-identity element of $G$. A base for $G$ is a sequence of points $(\omega_1,\dots,\omega_\ell)$ of $\Omega$ with trivial pointwise stabilizer, that is,
	\[
		G_{\omega_1,\dots,\omega_\ell} = 1.
	\] 
	The \textit{\textbf{base size}} of $G$, denoted by $b(G)$, is the smallest cardinality of a base for $G$. \\
	Both the base size and the minimal degree have been extensively studied, especially in the context of primitive groups. Except for their product, most results in the literature treat these two invariants separately. A simple yet fundamental inequality relating these quantities is the following: if $G$ is a transitive group of degree $n$, then
	\[
		\mu(G) b(G) \geq n,
	\]
	see for example~\cite[Excercise~$3.3.7$]{dixon_mortimer}.\\
	
	Known results on base size provide a natural upper bound for this product. For instance, it is shown in~\cite{MRD} that if $G$ is a primitive group of degree $n$ which is not large-base (i.e., not a wreath product action of two symmetric groups where the first acts on subsets), then
	\[
		b(G) \leq 2+\log n.
	\]
	Using this, we can easily obtain that the bound
	\[
		\mu(G) b(G) \leq n(2+\log n).
	\]
	holds for such primitive groups.\\
	Nonetheless, this bound is not optimal and admits numerous exceptions. A more refined estimate was provided in~\cite{fabio}, where a bound holds for all but one primitive group. Specifically, \cite[Theorem~$1.2$]{fabio} states that if 
	$G$ is a primitive group of degree $n$, different from the Mathieu group of degree $24$, then
	\[
		\mu(G) b(G) \leq n \log n.
	\]
	These results heavily rely on the deep structural knowledge of primitive groups, and much less is known for non primitive groups. \\This is a common fact: while results concerning base size and minimal degree abound for primitive groups, considerably less is known for transitive—and even less for intransitive—groups. Noteworthy in this context is \cite[Theorem~A]{KPS}, where a bound on the order of an arbitrary permutation group is given in terms of its minimal degree.\\
	The goal of this paper is to investigate the product $\mu(G)b(G)$ for imprimitive groups. In Section~\ref{sec:general}, we give evidences that a bound similar to the one of \cite[Theorem~$1.2$]{fabio} holds for some classes of transitive or even intransitive groups. In particular, Lemma~\ref{sec1:l1} shows that the bound holds for a wreath product group endowed with its imprimitive action, while Lemma~\ref{sec1:l2} establishes the bound for quasiprimitive groups, for Sylow subgroups of symmetric groups and for maximal subgroups of symmetric groups.\\
	
	Motivated by the results in Section~\ref{sec:general}, we initially conjectured that a bound similar to the one established in~\cite{fabio} would also hold for all transitive groups. However, to our surprise, this is not the case. Indeed, recall that since for every permutation group $G$ of degree $n$ both $\mu(G)$ and $b(G)$ are at most $n$, then $\mu(G)b(G) \leq n^2$. We show that $\mu(G)b(G)$ can, in fact, be arbitrarily close to $n^2$. Our main result provides a precise characterization of this behavior.
	\begin{theorem}\label{thrm:main}
		For every $\varepsilon>0$, there exists a transitive permutation group of degree $n$ such that
		$$\mu(G)b(G)\ge n^{2-\varepsilon}.$$
	\end{theorem}
	The proof of Theorem~\ref{thrm:main} is presented in Section~\ref{sec:proof}, and it relies on the descending series of a wreath product of two vector spaces over the field with $p$ element, for some prime number $p$.
\section{The bound for some permutation groups}\label{sec:general}

\begin{lemma}
	\label{sec1:l1}	
	Let $G = H \wr T$ be a finite transitive permutation group acting on a set $\Omega$ of size $n$ where $H$ is a primitive permutation group acting on a set $\Sigma$ which is a block for $G$ in $\Omega$ and where $T$ is a nontrivial transitive permutation group acting on the system of imprimitivity defined by $\Sigma$. Then $\mu(G) b(G) \leq n \log n$. 	
\end{lemma}

\begin{proof}
	Put $n = |\Omega|$ and $k = |\Sigma|$. We have $b(G) \leq b_{\Sigma}(H) (n/k)$ and $\mu(G) \leq \mu_{\Sigma}(H)$ where $b_{\Sigma}(H)$ denotes the minimal base size of $H$ acting on $\Sigma$ and $\mu_{\Sigma}(H)$ denotes the minimal degree of $H$ acting on $\Sigma$. If $H$ is different from the $5$-transitive Mathieu group of degree $24$, then $\mu(G)b(G) \leq (k \log k)(n/k) = n \log k \leq n \log n$, otherwise $\mu(G)b(G) \leq (16 \cdot 7) (n/24) \leq n \log 48 \leq n \log n$ by~\cite[Theorem~$1.2$]{fabio}
\end{proof}	

\begin{lemma}\label{sec1:l2}
	Let $G$ be a permutation group acting on a finite set $\Omega$ of size $n$. Let $G$ be different from the $5$-transitive Mathieu group of degree $24$. If $G$ 
	\begin{enumerate}
	\item is quasiprimitive, or 
	
	\item is a Sylow subgroup of $ \mathrm{Sym}(\Omega) $ or 
	
	\item is a maximal subgroup therein, 
	\end{enumerate} 	
	then $\mu(G) b(G) \leq n \log n$.
\end{lemma}	

\begin{proof}
	First let $G$ be transitive. If $G$ is primitive, then the claim follows from~\cite[Theorem~$1.2$]{fabio}. Let $G$ be different from a primitive group. Let $G$ be quasiprimitive. Let $\mathcal{B}$ be a maximal system of blocks in $\Omega$ defined from a nontrivial block $\Sigma$. Let $k = |\Sigma|$. Since $G$ is quasiprimitive, $b(G) \leq b_{\mathcal{B}}(G)$ and $\mu(G) \leq \mu_{\mathcal{B}}(G) k$. This gives $\mu(G)b(G) \leq (n/k) \log(n/k) k$ by~\cite[Theorem~$1.2$]{fabio}, which is at most $n \log n$. This proves (1).
	
	We continue to assume that $G$ is transitive. If $G$ is a (transitive) Sylow $p$-subgroup of $\mathrm{Sym}(\Omega)$ for some prime $p$, then (2) follows from Lemma \ref{sec1:l1} by taking $H$ to be the cyclic group of order $p$ and $T$ a Sylow $p$-subgroup of the symmetric group of degree $n/p$. If $G$ is a maximal imprimitive subgroup of $\mathrm{Sym}(\Omega)$, then $\mu(G) b(G) \leq n \log n$, again by Lemma~\ref{sec1:l1} by taking both $H$ and $T$ to be symmetric groups. This completes the proof of (3) in case $G$ is transitive. 
	
    Let $G$ act intransitively on $\Omega$. Let $G = G_{1} \times \cdots \times G_r$ be a Sylow subgroup of $\mathrm{Sym}(\Omega)$ where each $G_i$ is a transitive Sylow subgroup of $\mathrm{Sym}(\Omega_{i})$ for subsets $\Omega_i$ of $\Omega$ partitioning $\Omega$. We may assume that each $G_i$ is nontrivial. We have $\mu(G) \leq \min_{i} \{ \mu_{\Omega_i}(G_{i}) \}$ and $b(G) \leq \sum_{i=1}^{r} b_{\Omega_{i}}(G_{i})$. This and the previous paragraph give $\mu(G)b(G) \leq \sum_{i=1}^{r} |\Omega_i| \log |\Omega_i| \leq \sum_{i=1}^{r} |\Omega_i| \log n = n \log n$. This completes the proof of (2). Similarly, if $G$ is maximal (and intransitive) in $\mathrm{Sym}(\Omega)$, then $\mu(G)b(G) = 2 (n-2) \leq n \log n$, completing the proof of (3). 
	\end{proof}

\section{Proof of Theorem~\ref{thrm:main}}\label{sec:proof}
Let $a$ be a positive integer, let $p$ a prime number, and let $V$ be an $a$-dimensional vector space over the field  $\mathbb{F}_p$ with $p$ elements. We regard $V$ as a transitive regular subgroup of the symmetric group $\mathrm{Sym}(V)$. 

Next, we let $W = \mathbb{F}_p \mathrm{wr} V$ and $B = \mathbb{F}_p^V$ be the base group of the wreath product $W$. We regard $W$ as a transitive subgroup of the symmetric group $\mathrm{Sym}(\mathbb{F}_p \times V)$ endowed with its imprimitive action of degree $n=|\mathbb{F}_p\times V|=p^{a+1}$.

We let $B_0 = B$ and, for each positive integer $i$, we define recursively
\[
B_i = [B_{i-1}, V].
\]
An element of $B_0 = \mathbb{F}_p^V$ is a function from $V$ to $\mathbb{F}_p$. For what follows, it is useful to identify $B_0$ with a certain coordinate ring.

Let $X_1, \ldots, X_a$ be indeterminates, and consider the polynomial ring $\mathbb{F}_p[X_1, \ldots, X_a]$ with coefficients in $\mathbb{F}_p$. Now consider the evaluation map
\[
\mathbb{F}_p[X_1, \ldots, X_a] \to \mathbb{F}_p^V=B
\]
that sends a polynomial $f(X_1, \ldots, X_a)$ to the function which maps each $v = (v_1, \ldots, v_a) \in V=\mathbb{F}_p^a$ to $f(v_1, \ldots, v_a)\in \mathbb{F}_p$. This is a surjective map whose kernel is the ideal
\[
(X_1^p - X_1, \ldots, X_a^p - X_a).
\]
We write $x_i$ for the image of $X_i$ in the quotient ring, and identify $B_0$ with the coordinate ring $\mathbb{F}_p[x_1, \ldots, x_a]$. In particular, each element $f$ of $B_0$ is a polynomial function
\begin{align}\label{eq:1}
f = \sum_{\lambda_1,\ldots,\lambda_a=0}^{p-1} a_{\lambda_1,\ldots,\lambda_a} x_1^{\lambda_1}\cdots x_a^{\lambda_a},
\end{align}
where $a_{\lambda_1,\ldots,\lambda_a}\in\mathbb{F}_p$ for each $a$tuple $(\lambda_1,\ldots,\lambda_a)$. For not making the notation too cumbersome, given an $a$tuple $\bar{\lambda}=(\lambda_1,\ldots,\lambda_a)$, we denote with $x_{\bar{\lambda}}$ the monomial $\prod_i x_i^{\lambda_i}$.

Now, let $f \in B_0$ be as in~\eqref{eq:1} and let $v \in V$. Then,
\[
[f, v] =-f+ f^v =\sum_{\bar{\lambda}}  a_{\bar{\lambda}}(x_{\bar{\lambda}}^v - x_{\bar \lambda}).
\]
Observe that
\begin{align*}
(x_{\bar \lambda}^v - x_{\bar \lambda})(w_1, \ldots, w_a) = x_{\bar\lambda}(v_1 + w_1, \ldots, v_a + w_a) - x_{\bar\lambda}(w_1, \ldots, w_a)\\
=\prod_{i}(v_i+w_i)^{\lambda_i}-\prod_{i}w^{\lambda_i},
\end{align*}
and hence,
\[
x_{\bar \lambda}^v - x_{\bar \lambda} = \prod_{i}(x_i + v_i)^{\lambda_i} - x_{\bar \lambda}.
\]
Expanding the product $\prod_{i}(x_i + v_i)^{\lambda_i}$, we see that the term $\prod_{i} x_i^{\lambda_i}$ cancels out with $x_{\bar \lambda}$. Therefore, the commutator of an element $f \in B_0$ of degree $d$ and an element $v\in  V$ is an element of $B_0$ degree at most $d-1$. From this, arguing inductively, it immediately follows that
\[
B_d \subseteq \{f \in B_0 \mid \deg f \le a(p-1) - d\}.
\]
Observe now that, for every $1\le i \le a$, we have
\begin{align}\label{eq:2}
x_{\bar{\lambda}+e_i} =  \prod_{j \neq i} x_j^{\lambda_j}(x_i+1) - \prod_{j=1}^i x_j^{\lambda_j} = x_{\bar{\lambda}}^{e_i}-x_{\bar{\lambda}},
\end{align}
where $\{e_1,\ldots,e_a\}$ is the canonical basis of $V$. Now,~\eqref{eq:2} and an elementary induction yields 
\begin{align}\label{eq:0}B_d=\{f\in B\mid \deg f\le a(p-1)-d\}.
\end{align}

Observe that $B_{a(p-1)}$ consists of all constant functions. We have
$$0=B_{a(p-1)+1}<B_{a(p-1)}<\cdots <B_1<B_0=B=\mathrm{F}_p^V.$$
For every $d\in \{0,\ldots,a(p-1)\}$, we let $$G_d=B_{a(p-1)-d}\rtimes V.$$
In particular, $G_0=B_{a(p-1)}\rtimes V=B_{a(p-1)}\times V$ acts regularly on its domain. 
Observe that if $b_1\le b_2$, then $G_{b_1}\ge G_{b_2}$. In particular, $G_b$ is transitive for every $b\in \{0,\ldots,a(p-1)\}$.

\begin{lemma}\label{l:1}Let $b\in \{0,\ldots,a(p-1)\}$ and write $b=r(p-1)+s$, where $r,s\in\mathbb{N}$ and $0\le s< p-1$. We have $\mu(G_b)=(p-s)p^{a-r}$. In particular, if $b=r(p-1)$, then $\mu(G_{r(p-1)})=p^{a-r+1}$.
\end{lemma}
\begin{proof}
Let $g\in G_b$ such that the support of $g$ has cardinality $\mu(G_b)$. As $g\in B_0\rtimes V$, we may write $g=fv$, for some $v\in V$ and for some $f\in B_0$ with $\deg f\le b$, by~\eqref{eq:0}. If $v\ne 0$, then $g$ acts fixed point freely. Therefore, we may suppose that $v=0$ and $g=f\in B_{a(p-1)-b}$. 

Let $Z(f)=\{v\in V\mid f(v)=0\}$. Now, let $(x,v)$ be in the domain of $G_b$, from the definition of the wreath product we deduce that
$$(x,v)^f=x+f(v).$$
In particular, if $v\in Z(f)$, then $g=f$ fixes all the points of the form $(x,v)$; whereas, if $v\in V\setminus Z(f)$, then $f$ acts as a cycle of length $p$ on $\{(x,v)\mid x\in\mathbb{F}_p\}$. 
This shows that $$\mu(G)=p(|V\setminus Z(f)|).$$
Now, the result follows from~\cite[Theorem~5.11]{ffields}. 
\end{proof}

Next, we compute the cardinality of $|G_b|$. To this end, let $x$ be an indeterminate and consider the polynomial
$$p(x)=(1+x+\cdots+x^{p-1})^a\in \mathbb{Z}[x].$$
This polynomial has degree $(p-1)a$. Actually, the polynomial $p(x)$ enumerates something very important for our example. Let
$$p(x)=\sum_{k=0}^{(p-1)a}{a\choose k}^{(p-1)}x^k$$
be the expansion of $p(x)$ in its nomomials. The coefficients ${a\choose k}^{(p-1)}$ are usually called the extended binomial coefficients or multinomial coefficients. They do not have a standard notation, and we use the notation from~\cite{binomial}. When $p=2$, the extended binomial coefficients are equal to the usual binomial coeffiecients. We also give another example
$$(1+x+x^2)^4=x^8 + 4x^7 + 10x^6 + 16x^5 + 19x^4 + 16x^3 + 10x^2 + 4x + 1.$$
In particular, ${4\choose 6}^{(2)}={4\choose 2}^{(2)}=10$ and ${4\choose 4}^{(2)}=19$.

From the definition of $p(x)$, we see that ${a\choose k}^{(p-1)}$ counts the number of $a$tuples $(\lambda_1,\ldots,\lambda_a)$ with $k=\sum_{i=0}^{a}\lambda_i$ and with $0\le \lambda_i\le p-1$, $\forall i\in \{0,\ldots,a\}$. Thus,~\eqref{eq:0} gives
$$\dim_{\mathbb{F}_p}(B_{a(p-1)-k}/B_{a(p-1)-k+1})={a\choose k}^{(p-1)}.$$
Therefore,
$$\dim_{\mathbb{F}_p}B_{a(p-1)-b}=\sum_{k=0}^b{a\choose k}^{(p-1)}.$$
\begin{lemma}\label{l:2}For every $b\in \{0,\ldots,a(p-1)\}$, $b(G_b)=\sum_{k=0}^b{a\choose k}^{(p-1)}$.
\end{lemma}
\begin{proof}
The stabilizer of a point in $G_b=B_{a(p-1)-b}\rtimes V$ is a subgroup of $B_{a(p-1)-b}$ having index $p$, because the degree of the action is $n=p^{a+1}$ and $|V|=p^a$. Now, as all the orbits of $B_{a(p-1)-b}$ have cardinality $p$, we deduce that we need to fix $\dim_{\mathbb{F}_p}B_{a(p-1)-b}-1$ more points to obtain a basis.
\end{proof}

Now, fix $r\in \{0,\dots,a\}$. From Lemmas~\ref{l:1} and~\ref{l:2}, we have
\begin{equation}\label{eq:422}
\mu(G_{r(p-1)})b(G_{r(p-1)})=p^{a-r+1}\sum_{k=0}^{r(p-1)}{a\choose k}^{(p-1)}.
\end{equation}

Before dealing with the general case, we use~\eqref{eq:422} to make an explicit computation when $p=2$. Recall that, when $p=2$, we have ${a\choose k}^{(p-1)}={a\choose k}$. From~\cite[Exercise~9.42, page~492]{knuth}, we have
\begin{align*}
\sum_{k=0}^{r}{a\choose k}=2^{a\left(\lambda\log_2(1/\lambda)+(1-\lambda)\log_2\left(\frac{1}{1-\lambda}\right)\right)-\log_2(a)+O(1)},
\end{align*}
where $\lambda=r/a$.
Thus, from~\eqref{eq:42}, we get
\begin{align*}
b(G_r)\mu(G_r)&=2^{a(1-\lambda)+a\left(\lambda\log_2(1/\lambda)+(1-\lambda)\log_2\left(\frac{1}{1-\lambda}\right)\right)-\log_2(a)+O(1)}.
\end{align*}
As the degree of the permutation group is $n=2^{a+1}$, we deduce
\begin{align}\label{eq:limit}
\lim_{a\to \infty}\frac{\log_2(b(G_r)\mu(G_r))}{\log_2 (n)}&=1-\lambda-\lambda\log_2(\lambda)-(1-\lambda)\log_2\left(1-\lambda\right).
\end{align}
It is elementary to show that the maximum of the function appearing on the right hand side is attained when $\lambda=1/3$. With $\lambda=1/3$, the limit in~\eqref{eq:limit} equals $\log(3)/\log(2)$. Therefore, we have proved the following.
\begin{lemma}\label{l:lemmanew}
For every $\varepsilon>0$, there exists a transitive permutation $2$-group of degree $n$ such that
$\mu(G)b(G)\ge n^{\log(3)/\log(2)-\varepsilon}$.
\end{lemma}

We now turn to the general case. From~\eqref{eq:422}, we have
\begin{equation}\label{eq:42}
\mu(G_{r(p-1)})b(G_{r(p-1)})\ge p^{a-r+1}{a\choose r(p-1)}^{(p-1)}.
\end{equation}
\begin{proof}[Proof of Theorem~$\ref{thrm:main}$]
Assume $p>3$, $a$ is a multiple of $p-1$ and write $a=(p-1)a_1$. Let $c=\lfloor \sqrt{p}\rfloor$.
From~\cite[Theorem~5]{binomial2}, we deduce that, for $p>3$,
$${a\choose ca}^{(p-1)}$$
is asymptotic to 
$$\frac{\phi(x)}{\sqrt{2\pi a}}\left(\frac{1-x^p}{x-x^2}\right)^a,$$
as $a$ tends to infinity, where
\begin{align*}
\phi(x)&=\left(\frac{x}{(1-x)^2}-\frac{p^2x^p}{(1-x^p)^2}\right)^{-1/2},&& x=\frac{1}{d}+\frac{p(d-1)^2}{d^{p+2}}+\theta\frac{p^3}{d^{2p}},\quad d=1+\frac{1}{c},\quad |\theta|\le 1.
\end{align*}
In~\cite{binomial}, there is a much more informative asymptotic estimate on the extended binomial coefficients, but only for certain very special values of $c$.

Observe that  $\phi(x)$ depends only on $p$, but not on $a$.
Observe now that, from Lemmas~\ref{l:1} and~\ref{l:2} and from~\eqref{eq:42}, we have
\begin{align*}
\lim_{a\to\infty}\log_n(\mu(G_{ca})b(G_{ca}))&=\lim_{a\to\infty}\frac{\log_p(\mu(G_{ca})b(G_{ca}))}{\log_p(n)}\\
&\ge\lim_{a\to\infty}\frac{a-ca/(p-1)+1+\log_p(\phi(x)/\sqrt{2\pi a})+a\log_p((1-x^p)/(x-x^2))}{a+1}\\
&=1-\frac{\lfloor\sqrt{p}\rfloor}{p-1}+\lim_{a\to\infty}\frac{\log_p(\phi(x)/\sqrt{2\pi a})}{a+1}+\log_p((1-x^p)/(x-x^2))\\
&=1-\frac{\lfloor\sqrt{p}\rfloor}{p-1}+\log_p((1-x^p)/(x-x^2)),
\end{align*}
observe that $\log_p(\phi(x)/\sqrt{2\pi a})/(a+1)\to 0$ because, as we remarked above, $x$ does not depend on $a$.

Observe now that
\begin{align*}
\lim_{p\to \infty}\frac{p(d-1)^2}{d^{p+2}}&=\lim_{p\to\infty}\frac{1}{d^{p+2}}=\lim_{p\to\infty}\left(1+\frac{1}{c}\right)^{-p-2}=\lim_{p\to\infty}\left(\left(1+\frac{1}{c}\right)^{c}\right)^{-(p+2)/c}\\
&=\lim_{p\to\infty}e^{-p/\sqrt{p}}=0.
\end{align*}
Moreover, with an analogue argument, we deduce
\begin{align*}
\lim_{p\to \infty}\frac{p^3}{d^{2p}}&=\lim_{p\to\infty}\frac{p^3}{e^{2p/\sqrt{p}}}=0.
\end{align*}
This shows that $\lim_{p\to\infty}x=1$ and hence 
$$\lim_{p\to\infty}\log_p\left(\frac{1-x^p}{x-x^2}\right)=\lim_{p\to\infty}\log_p\left(\frac{1+x+\cdots +x^{p-1}}{x}\right)=1.$$

This gives
$$\lim_{p\to\infty}\lim_{a\to\infty}\log_n(\mu(G_{ca})b(G_{ca}))=2.\qedhere$$
\end{proof}

\thebibliography{10}
\bibitem{dixon_mortimer}J.~D.~Dixon,  B.~Mortimer, \textit{Permutation groups}, Graduate Texts in Mathematics \textbf{163}, Springer-Verlag, New York, 1996
\bibitem{knuth}R.~Graham, D.~Knuth, O.~Patashnik, \textit{Concrete Mathematics}, second edition, Pearson Education Limited, 1994.. 
\bibitem{ffields}X.~Hou, \textit{Lectures on Finite Fields}, Graduate studies in Mathematics \textbf{190}, AMS, 2018.
\bibitem{KPS}J.~Kempe, L.~Pyber, A.~Shalev, 
Permutation groups, minimal degrees and quantum computing,
\textit{Groups Geom. Dyn.} \textbf{1} (2007), 553--584.
\bibitem{binomial2}J.~Li, Asymptotic Estimate for the
Multinomial Coefficients, \textit{J. Integer Sequences} \textbf{23} (2020), Article~20.1.3.
\bibitem{fabio}F.~Mastrogiacomo, On the minimal degree and base size of finite primitive groups, \textit{J. Algebra and its applications}.
\bibitem{MRD}M.~Moscatiello, C.~Roney-Dougal, Base sizes of primitive permutation groups, \textit{Monatsh. Math.} \textbf{198}
(2022), 411-443.
\bibitem{binomial}T.~Neuschel, A note on extended binomial coefficients, \textit{J. Integer Sequences} \textbf{17} (2014), Article~14.10.4.
\end{document}